\documentclass[10pt,reqno]{amsart}
\usepackage{amsmath, amssymb, multicol, url, thmtools, thm-restate, enumitem}

\allowdisplaybreaks[2]

\theoremstyle{plain}
\newtheorem{theorem}{Theorem}
\newtheorem{lemma}[theorem]{Lemma}
\newtheorem{corollary}[theorem]{Corollary}

\theoremstyle{remark}
\newtheorem*{remark*}{Remark}

\begin{document}

\title[On the order sequence of an embedding of the Ree curve]{On the order sequence of an embedding of the Ree curve}

\author{Dane C. Skabelund}
\address{
Department of Mathematics\\
University of Illinois\\
Urbana, IL 61801\\
U.S.A.}
\email{skabelu2@illinois.edu}

\begin{abstract}
  In this paper we compute the Weierstrass order-sequence associated with a certain linear series on the Deligne-Lusztig curve of Ree type.
  As a result, we determine that the set of Weierstrass points of this linear series consists entirely of $\mathbb F_q$-rational points.
\end{abstract}

\maketitle

\section{Introduction}

Let $X$ be a smooth, geometrically irreducible, projective algebraic curve defined over a finite field $\mathbb F_q$ of characteristic $p$, and let
\[
m(t) = t^n + a_{n-1}t^{n-1} + \cdots + a_1t + a_0 \in \mathbb{Z}[t]
\]
be the square-free part of the characteristic polynomial of the Frobenius endomorphism $\operatorname{Fr}_q$ on the Jacobian of $X$.
Then for any $P, P_0 \in X$ with $P_0$ an $\mathbb F_q$-rational point, we have the fundamental linear equivalence \cite{HKT}
\begin{equation} \label{fundequiv}
  m(\operatorname{Fr}_q)(P)
  = \operatorname{Fr}_q^n(P) + \cdots + a_1\operatorname{Fr}_q(P) + a_0 P
  \sim m(1) P_0 .
\end{equation}
Thus for $m = |m(1)|$, the linear series $\mathcal D_X := |mP_0|$, sometimes called the \emph{Frobenius linear series}, is independent of the choice of rational point $P_0$, and is completely determined by the zeta function $Z_X(t)$.

The linear series $\mathcal D_X$ is a useful tool for studying curves with many rational points.
It has been used to study \emph{$\mathbb F_q$-maximal curves}, that is, curves defined over $\mathbb F_q$ whose number of $\mathbb F_q$-rational points attains the Hasse-Weil bound \cite{FGT97}, \cite{AT05}, \cite{FG09}, \cite{FG10}, as well as \emph{$\mathbb F_q$-optimal curves}, which have the greatest number of $\mathbb F_q$-rational points among curves of their genus \cite{FT98}.

Notable among the latter group are the Hermitian, Suzuki, and Ree curves, which are the Deligne--Lusztig curves associated to the simple groups ${}^2A_2$, ${}^2B_2$, and ${}^2G_2$.
These curves satisfy strong uniqueness properties.
Each is characterized among curves over $\mathbb F_q$ by its genus, number of rational points, and automorphism group \cite{HP93}.
Moreover, it can be shown that the genus and number of rational points alone are sufficient to characterize the Hermitian and Suzuki curves \cite{RS94}, \cite{FT98}.
Whether this is also the case for the Ree curve remains an open question---one  which  was the initial motivation for the work in the current paper.

For fixed $s \geq 1$, let $q_0 = 3^s$ and $q = 3^{2s+1}$.
The Ree curve $X = X(q)$ over $\mathbb F_q$ has a singular affine model given by the two equations
\begin{equation}\label{yz}
 y^q - y = x^{q_0}(x^q-x), \qquad z^q - z = x^{q_0}(y^q-y) ,
\end{equation}
and has genus $g = \frac 32q_0(q-1)(q+q_0+1)$ and $N = q^3+1$ points defined over $\mathbb F_q$.
Weil--Serre's explicit formulas can be used to show that $X$ is $\mathbb F_q$-optimal, and that any curve defined over $\mathbb F_q$ with this $g$ and $N$ has $L$-polynomial
\begin{equation}\label{Lpoly}
  L_X(t)
  = (1+3q_0t+qt^2)^{q_0(q^2-1)}(1+qt^2)^{\frac 12 q_0(q-1)(q+3q_0+1)} .
\end{equation}
Since the characteristic polynomial of $\operatorname{Fr}_q$ is $t^{2g} L_X(1/t)$, we obtain
\begin{align*}
  m(t)
  &= (t^2+3q_0t+q)(t^2+q) \\
  &= t^4 + 3q_0t^3 + 2qt^2 + 3qq_0t  + q^2, 
\end{align*}
and $\mathcal D_X = |mP_0|$ with $m = m(1) = 1 + 3q_0 + 2q + 3qq_0 + q^2$.

There is a subseries $\mathcal D \subset \mathcal D_X$ of projective dimension 13 which is invariant under $\operatorname{Aut}(X)$.
In \cite{DE15}, Duursma and Eid show that $\mathcal D$ is very ample, giving a smooth embedding of $X$ in $\mathbb P^{13}$.
They also find 105 equations describing the image of this embedding, and use these to compute the Weierstrass semigroup at a rational point when $s=1$.
In this case, it follows from their work that $\mathcal D = \mathcal D_X$ is a complete linear series.
Whether or not $\mathcal D$ is complete for $s \geq 2$ is unknown at present.
In \cite{Kane16}, Kane also gives an embedding of $X$ in $\mathbb P^{13}$ that arises from the abstract theory of Deligne--Lusztig varieties. The exact relationship between the two embeddings is not immediately clear.

In this paper we determine the order sequence of $\mathcal D$, that is, the orders of vanishing of sections of $\mathcal D$ at a general point.
Equivalently, these are the intersection multiplicities of hyperplane sections of $X$ embedded in $\mathbb P^{13}$ at a general point.
We prove the following theorem.
\begin{restatable}{theorem}{Dthm}
\label{Dthm}
The orders of $\mathcal D$ are
\[
  0,\; 1,\; q_0,\; 2q_0,\; 3q_0,\; q,\; q+q_0,\; 2q,\; qq_0,\; qq_0+q_0,\; qq_0+q,\; 2qq_0,\; 3qq_0,\; q^2 .
\]
Since $\mathcal D \subset \mathcal D_X$, these form a subset of the orders of $\mathcal D_X$.
\end{restatable}
As a consequence, we show in the final section that the Weierstrass points of $\mathcal D$ consist of the $\mathbb F_q$-rational points of $X$.\\

\emph{Acknowledgements:} I would like to thank the anonymous referee for his/her detailed comments which improved the exposition of this paper.

\section{Background}

The theory of Weierstrass points in characteristic $p$ was developed first by F.K. Schmidt \cite{Schmidt}.
We briefly give the necessary definitions and results on the subject following the presentation in the paper of St\"ohr and Voloch \cite{SV86}.

Given a base-point-free linear series $\mathcal D$ on $X$ of dimension $r$ and degree $d$, and $P$ a point of $X$, the \emph{$(\mathcal D,P)$-orders} consist of the sequence
\[
 0 = j_0(P) < j_1(P) < \cdots < j_r(P) \leq d
\]
of integers $j_i$ such that there is a hyperplane in $\mathcal D$ intersecting $P$ with multiplicity equal to $j_i$.
These are the same for all but finitely many points $P \in X$, called $\mathcal D$-\emph{Weierstrass points}.
The generic values of the $j_i(P)$ are the \emph{$\mathcal D$-orders}
\[
 0 = \epsilon_0 < \epsilon_1 < \cdots < \epsilon_r .
\]
This order sequence may be computed by choosing the $\epsilon_i$ lexicographically smallest so that
\[
  ( D_x^{\epsilon_i} f_0: D_x^{\epsilon_i} f_1 : \cdots : D_x^{\epsilon_i} f_r),
  \qquad i = 1,\ldots,r,
\]
are linearly independent in $\mathbb P^r_{\mathbb F_q(X)}$, where $f_0,f_1,\ldots,f_r$ is a basis for $\mathcal D$, and the $D_x^i$ are Hasse derivatives taken with respect to some fixed separating variable $x$.

The Hasse derivatives $D_x^i$ are defined on $\mathbb F_q(x)$ by
\[
  D_x^i x^j = \binom ji x^{j-i},
\]
and extend to derivations on $\mathbb F_q(X)$ satisfying the properties
\[
  D_x^k(fg) = \sum_{i+j=k} (D_x^if)( D_x^j g)
  \qquad\text{and}\qquad
D_x^k f^p = \begin{cases}
               (D_x^{k/p}f)^p & \text{if } p \mid k  \\
	       0 & \text{otherwise}
              \end{cases}
\]
for any $f,g \in \mathbb F_q(X)$.
In view of this second property, it will be often be convenient to write $D_x^{k/q}$ for $k/q$ is a rational number with denominator a power of $p$, adopting the convention that $D_x^{k/q} = 0$ when $k/q$ is not an integer.
Furthermore, when the choice of separating variable $x$ is clear from context, we omit the subscript and write simply $D^i$.

The following ``$p$-adic criterion'' for $\mathcal D$-orders is quite useful.
\begin{lemma}[\cite{SV86}, Corollary 1.9]
  If $\epsilon$ is a $\mathcal D$-order and $\binom{\epsilon}{\mu} \not\equiv 0 \bmod p$, then $\mu$ is also a $\mathcal D$-order.
\end{lemma}

By Lucas's Theorem, the condition $\binom{\epsilon}{\mu} \not\equiv 0 \bmod p$ in the lemma is equivalent to saying that the coefficients in the $p$-adic expansion of $\epsilon$ are greater than or equal to those in the expansion of $\mu$.
When this is the case we write $\mu \leq_p \epsilon$.
This defines a partial order on the nonnegative integers.

The ($q$-)\emph{Frobenius orders} $0 = \nu_0 < \nu_1 < \cdots < \nu_{r-1}$ of $\mathcal D$ form a subsequence of the order sequence $\{\epsilon_i\}$, and are defined lexicographically smallest so that
\begin{align*}
  &(f_0^q:f_1^q:\cdots:f_r^q), \\
  &( D_x^{\nu_0} f_0: D_x^{\nu_0} f_1 : \cdots : D_x^{\nu_0} f_r), \\
  &\qquad\vdots \\
  &( D_x^{\nu_{r-1}} f_0: D_x^{\nu_{r-1}} f_1 : \cdots : D_x^{\nu_{r-1}} f_r)
\end{align*}
are linearly independent in $\mathbb P^r_{\mathbb F_q(X)}$.
There is exactly one $\mathcal D$-order $\epsilon_I$ which is omitted by the sequence $\{\nu_i\}$.
The geometric significance of the index $I$ is as follows:
it is the smallest $i \geq 0$ such that, for general $P$, the image of $P$ under the Frobenius endomorphism lies in the $i$th osculating space at $P$.
The Frobenius orders are closely connected with the $\mathbb F_q$-rational points of $X$, and are used in St\"ohr and Voloch's proof of the Riemann Hypothesis for curves over finite fields.

\begin{lemma}[\cite{SV86}, discussion preceding Proposition 2.3]\label{smallnu}
Let $(1:f_1:\cdots : f_r)$ be the morphism associated to $\mathcal D$.
Then the Frobenius orders of $\mathcal D$ which are less than $q$ are the first several orders of the morphism $(f_1-f_1^q: \cdots : f_r-f_r^q)$.
\end{lemma}

\section{Derivatives on the Ree Curve}

The function field of the Ree curve $X$ is $\mathbb F_q(x,y,z)$, where $y$ and $z$ satisfy \eqref{yz}.
The linear series $\mathcal D$ we wish to study corresponds to the $\overline{\mathbb F}_q$-vector space $V_{\mathcal D}$ spanned by the 14 functions
\[
\mathcal B
= \{
1,\;
x,\;
y,\;
z,\;
w_1,\;
w_2,\;
\ldots,\;
w_{10} \} ,
\]
where $w_i$ are defined by
\begin{align*}
w_1 &= x^{3q_0+1} - y^{3q_0}    &w_6 &= v^{3q_0} - w_2^{3q_0} + x w_4^{3q_0}\\
w_2 &= xy^{3q_0} - z^{3q_0} &w_7 &= w_2+v\\
w_3 &= xz^{3q_0} - w_1^{3q_0} &w_8 &= w_5^{3q_0}+xw_7^{3q_0}
\addtocounter{equation}{1}\tag{\theequation} \label{wi} \\*
w_4 &= xw_2^{q_0} - yw_1^{q_0} &w_9 &= w_4w_2^{q_0}-yw_6^{q_0}\\
v &= xw_3^{q_0} - zw_1^{q_0} & w_{10} &= zw_6^{q_0} - w_3^{q_0}w_4\\
w_5 &= yw_3^{q_0} - zw_1^{q_0}
\end{align*}
as in the appendix of \cite{Pedersen91}.
The functions in $\mathcal B$ have distinct orders at the pole $P_\infty$ of $x$, hence are linearly independent.
We will use the separating variable $x$ for computing all Hasse derivatives on the Ree curve.

To compute the orders of $\mathcal D$, we will need to obtain closed form expressions for the derivatives of the functions $f \in \mathcal B$.
In addition to the relations in \eqref{yz}, the following equations derived by Pedersen will be useful for computing the derivatives of the $w_i$.
\begin{align*}
  w_1^q-w_1 &= x^{3q_0}(x^q-x)     & w_4^q-w_4 &= w_2^{q_0}(x^q-x) - w_1^{q_0}(y^q-y) \\*
  w_2^q-w_2 &= y^{3q_0}(x^q-x)     & w_5^q-w_5 &= w_3^{q_0}(y^q-y) - w_2^{q_0}(z^q-z) \\*
  w_3^q-w_3 &= z^{3q_0}(x^q-x)     & w_7^q-w_7 &= w_2^{q_0}(y^q-y) - w_3^{q_0}(x^q-x) \addtocounter{equation}{1}\tag{\theequation} \label{ped} \\*
  w_6^q-w_6 &= w_4^{3q_0}(x^q-x)   & w_9^q-w_9 &= w_2^{q_0}(w_4^q-w_4) - w_6^{q_0}(y^q-y) \\*
  w_8^q-w_8 &= w_7^{3q_0}(x^q-x)   & w_{10}^q-w_{10} &= w_6^{q_0}(z^q-z) - w_3^{q_0}(w_4^q-w_4)
\end{align*}
We have separated these equations into groups of similar form.
We call the $w_i$ which appear on the left hand side of \eqref{ped} of type 1 and the $w_i$ on the right hand side of type 2.

We give an example to show how the expressions in \eqref{ped} are useful for computing derivatives.
To compute the derivatives of $y$, we let $h = x^{q_0}(x^q-x)$.
Since $y^q-y = h$, we may expand $y$ as a series in $h$ whose tail is contained in the kernel of any of the derivations we wish to apply.
To compute $D^i y$ for $i < q^2$, we consider
\[
  y = - h - h^q + y^{q^2} \equiv - h - h^q \mod \overline{\mathbb F}_q(X)^{q^2},
\]
since $\overline{\mathbb F}_q(X)^{q^2} = \bigcap_{i=1}^{q^2-1} \ker D^i$.
Then
\[
  D^iy
  = -D^ih - (D^{i/q}h)^q .
\]
In this manner, the derivatives of $y$ are written in terms of derivatives of $h$, which can be determined using the basic properties of Hasse derivatives.

Each equation in \eqref{ped} is of a similar form, with each new function written in terms of previous ones.
Therefore, one may in principle write down any derivative $D^if$ with $f \in \mathcal B$ as an element of $\mathbb F_q[x,y,z]$ using this method.

For each $f \in \mathcal B$, we construct a set $S_f$ containing all indices $i$ in $\{0, 1, \ldots, q^2\}$ such that $D^if \neq 0$, which we refer to as the \emph{support} of $f$.
We make no claims that $D^if \neq 0$ for all $i$ in $S_f$.
By direct calculation as in the example above, the sets
\begin{align*}
  S_{x^q-x} &= \{ 0, 1, q \} \\
  S_{y^q-y} &= \{ 0, 1, q_0, q_0+1, q, q+q_0\}  \\
  S_y &= \{ 0, 1, q_0, q_0+1, q, q+q_0, qq_0, qq_0+q, q^2 \} \\
  S_{z^q-z} &= \{ 0, 1, q_0, q_0+1, 2q_0, 2q_0+1, q, q+q_0, q+2q_0 \} \\
  S_z &= \{ 0, 1, q_0, q_0+1, 2q_0, 2q_0+1, q, \\
       &\qquad q+q_0, q+2q_0, qq_0, qq_0+q, 2qq_0, 2qq_0+q, q^2 \}
\end{align*}
satisfy the desired conditions.

Now we construct $S_{w_i}$ for $i = 1,\ldots,10$.
For $n \geq 1$ and $A, B \subset \{ 0,\ldots, q^2 \}$ we use the notation
\begin{align*}
 nA &= \{ na : a \in A \} \cap [0,q^2] , \\
 A + B &= \{ a+b : a \in A, b \in B\} \cap [0,q^2] .
\end{align*}
Since $w_1^q-w_1 = x^{3q_0}(x^q-x)$, we define
\begin{align*}
  S_{w_1^q-w_1}
  &= 3q_0 S_x + S_{x^q-x} \\
  &= \{ 0, 3q_0 \} + \{ 0, 1, q\}
  = \{ 0, 1, 3q_0, 3q_0+1, q, q+3q_0\}
\end{align*}
and
\begin{align*}
 S_{w_1}
  &= S_{w_1^q-w_1} \cup q S_{w_1^q-w_1} \\
  &= \{ 0, 1, 3q_0, 3q_0+1, q, q+3q_0, 3qq_0, 3qq_0+q, q^2 \} .
\end{align*}
Similarly, we define $S_{w_2^q-w_2} = 3q_0 S_{y} + S_{x^q-x}$ and $S_{w_2} = S_{w_2^q-w_2} \cup q S_{w_2^q-w_2}$, and so on, using the equations in \eqref{ped} as a guide.

Let $S$ denote the union of the $S_f$ with $f \in \mathcal B$.
Since we will use this information later, we include a table of the $S_f$ in an appendix, and note in particular that
\[
  S_x
  \subset S_{w_1}
  \subset S_{w_2}
  \subset S_{w_3}
  \subset S_{w_6}
  = S_{w_8}
\]
and
\[
  S_y
  \subset S_z
  \subset S_{w_4}
  \subset S_{w_7}
  \subset S_{w_5}
  \subset S_{w_9}
  \subset S_{w_{10}} .
\]
For $s=1$, the indices appearing in $S$ as represented in the appendix are not all distinct, since for example $3q = qq_0$.
To avoid any complications this may cause, we assume going forward that $s \geq 2$.
Computations performed in Magma \cite{MAGMA} have verified the statements of all our results for $s=1$.

\section{Computation of Orders}

In this section we compute the orders of $\mathcal D$.

\Dthm*

\begin{lemma}\label{smalleps}
  The orders of $\mathcal D$ which are less than $q$ are $0$, $1$, $q_0$, $2q_0$, and $3q_0$.
\end{lemma}

\begin{proof}

  That $\epsilon_0(\mathcal D) = 0$ and $\epsilon_1(\mathcal D) = 1$ is clear.
  The rest follows from Lemma \ref{smallnu}.
  Indeed, the orders of the morphism
  \[
    (x^q-x: y^q-y: z^q-z: w_1^q-w_1 : w_2^q-w_2 : \cdots )
  = ( 1: x^{q_0}:  x^{2q_0}:  x^{3q_0} : y^{3q_0}: \cdots )
  \]
  which are less than $q$ are $0$, $q_0$, $2q_0$, and $3q_0$, so these are the Frobenius orders of $\mathcal D$ which are less than $q$.
  There is only one order of $\mathcal D$ which is not a Frobenius order, and this is $\epsilon_1(\mathcal D)$.
  \qedhere
\end{proof}

\begin{remark*}
In light of Lemma 3, the fact that $\nu_1(\mathcal D) = \epsilon_2(\mathcal D) > 1$ means that the matrix
\[
\begin{pmatrix}
  x^q-x & y^q-y & z^q-z & \cdots & w_{10}^q-w_{10}  \\
  1 & D^1 y & D^1 z & \cdots & D^1 w_{10}
\end{pmatrix}
\]
has rank 1, so that
\begin{equation}\label{nu1}
  f^q - f = (x^q-x) D^1f
\end{equation}
holds for all $f \in \mathcal B$.
Applying the derivatives $D^q$ and $D^{kq_0}$ for $k=1,2,3$ to the previous equation gives the identities
\begin{align}
 D^{kq_0}f &= -(x^q-x) D^{kq_0+1}f, \qquad k = 1,2,3  \label{kq0}  \\
 D^q(f^q-f) &= D^1f + (x^q-x) D^{q+1}f  \label{Dqbbq}
\end{align}
which hold for all $f \in \mathcal B$. These will be used extensively in what follows.
\end{remark*}

From \eqref{fundequiv}, we have for any $P \in X$ a linear equivalence
\[
  \operatorname{Fr}^4(P) + 3q_0\operatorname{Fr}^3(P) + 2q\operatorname{Fr}^2(P) + 3qq_0\operatorname{Fr}(P) + q^2 P \sim m P_\infty .
\]
If $P \not\in X(\mathbb F_q)$, then the terms on the left hand side involve distinct points since $X$ has no places of degrees 2, 3, or 4 over $\mathbb F_q$.
By applying some multiple of the Frobenius to this equivalence, we obtain each of $1$, $3q_0$, $2q$, $3qq_0$, and $q^2$ as orders of $\mathcal D_X$, as in Lemma 3.2 of \cite{FT98}.
By the $p$-adic criterion it follows that $q$ is also an order of $\mathcal D_X$.
That said, it is not immediately clear that these are orders of the linear series $\mathcal D$.
We show now that these are in fact the orders of the subseries $\mathcal E \subset \mathcal D$ corresponding to $V_{\mathcal E} = \overline{\mathbb F}_q \langle 1,x,w_1,w_2,w_3,w_6,w_8 \rangle$, and hence are orders of $\mathcal D$.

\begin{theorem}\label{Ethm}
The orders of $\mathcal E$ are $0$, $1$, $3q_0$, $q$, $2q$, $3qq_0$, and $q^2$.
\end{theorem}

For ease of notation we write $\ell = x^q-x$ in the proof of the next lemma and throughout the rest of the paper.

\begin{lemma}\label{hypersurface}
  The image in $\mathbb P^6$ of the map $\phi_{\mathcal E} = (1:x:w_1:w_2:w_3:w_6:w_8)$ lies on the hypersurface
  \[
    \sum_{i+j = 6} X_i^{q^2} X_j = 0 .
  \]
\end{lemma}

\begin{proof}
By using \eqref{ped} one finds that
\begin{align*}
  w_8^{q^2} + w_8
  &= (w_7^{3q_0})^q \ell^q + w_7^{3q_0} \ell - w_8 \\
  x w_6^{q^2} + x^{q^2} w_6
  &= ( (w_4^{3q_0})^{q}x + w_6 )\ell^q
   + ( w_4^{3q_0}x + w_6 )\ell
   - xw_6 \\
  w_1w_3^{q^2} + w_1^{q^2}w_3
  &= ( (z^{3q_0})^{q}w_1   + (x^{3q_0})^{q}w_3 )\ell^q
   + (  z^{3q_0}w_1        + x^{3q_0}w_3 )\ell
   - w_1w_3 \\
  w_2^{q^2+1}
  &= (y^{3q_0})^{q}w_2 \ell^q + y^{3q_0}w_2 \ell  + w_2^2 .
\end{align*}
Summing these and collecting terms involving common powers of $\ell$ gives an expression of the form
\[
  A_{-1} + A_0\ell + A_1\ell^q .
\]
That each $A_i = 0$ on $X$ may be verified using \eqref{wi} and \eqref{ped}, along with some of the 105 equations found in \cite{DE15}.
This calculation is carried out more explicit detail in the proof of Lemma 4.3 of \cite{rcf}.
\qedhere

\end{proof}

\begin{lemma}
  The largest order of $\mathcal E$ is $q^2$.
\end{lemma}

\begin{proof}

Because the coefficients $a_i$ of $m(t) = t^4 + 3q_0t^3 + 2qt^2 + 3qq_0t  + q^2$ satisfy $a_0 \geq a_2 \geq \cdots \geq a_4$ and $\#X(\mathbb F_q) > q(m-a_0)+1$, it follows from Proposition 3.4 of \cite{FGT97} that the largest order of $\mathcal D_X$ is $q^2$ (our numbering of the coefficients $a_i$ is opposite that found in the reference).
Since each order of $\mathcal E$ is an order of $\mathcal D_X$, no order of $\mathcal E$ is greater than $q^2$.
We exhibit a family of functions $g_P$ in $V_{\mathcal E}$ parameterized by $P$ in $X$ which vanish to order at least $q^2$ at $P$.

Write $(1,x,w_1,w_2,w_3,w_6,w_8) = (f_0,\ldots,f_6)$.
Then by Lemma \ref{hypersurface}, the function
\[
  G(P,Q) = \sum_{i+j = 6} f_i^{q^2}(P) f_j(Q)
\]
vanishes on the diagonal of $X \times X$.
Choose any $P \in X \smallsetminus \{P_\infty\}$.
Then $g_P = G(P,\cdot )$ and $h_P = G(\cdot ,P)$ are functions on $X$ which vanish at $P$, and $g_P$ is in $V_{\mathcal E}$.
Since
\[
  h_P = \sum_{i+j = 6} f_j(P) f_i^{q^2}
  = \left( \sum_{i+j = 6} f_j(P)^{1/q^2} f_i \right)^{q^2},
\]
the function $h_P$ vanishes at $P$ to order at least $q^2$.
But
\begin{align*}
  g_P - h_P
  &= \sum_{i+j=6} f_i^{q^2}(P)f_j - f_j(P) f_i^{q^2} \\
  &= \sum_{i+j=6} (f_i^{q^2}(P)-f_i^{q^2})(f_j(P) + f_j ) + f_i^{q^2}f_j - f_i^{q^2}(P)f_j(P) \\
  &= \sum_{i+j=6} (f_i(P)-f_i)^{q^2} (f_j(P) + f_j)
\end{align*}
also vanishes at $P$ to order at least $q^2$, hence so does $g_P$.
Since $P$ was chosen in an open subset of $X$, $q^2$ is an order of $\mathcal E$.
\qedhere

\end{proof}

\begin{remark*}
  The proof of the preceding lemma shows that
  \[
    \sum_{i+j = 6} f_i^{q^2}(P) X_j = 0
  \]
  is the equation of the osculating hyperplane at $P$, and that the function $X \to \operatorname{Div}(X)$ taking $P \mapsto \operatorname{div}(g_P)$ assigns to $P$ the corresponding hyperplane section.
\end{remark*}

\begin{proof}[Proof of Theorem \ref{Ethm}]

  Let $M$ be the matrix of derivatives $[D^i f_j]$ with $i \in \{ 0,1,3q_0+1,q+3q_0+1,2q+3q_0+1\}$ and $f_j \in \{ 1,x,w_1,w_2,w_3\}$.
  By the appendix, the matrix $M$ is upper triangular.
  Moreover, one may check by hand that each diagonal entry is equal to 1.
  Thus, there are at least 5 orders $\epsilon$ of $\mathcal E$ with $\epsilon \leq 2q+3q_0+1$.

  Let $I_{\mathcal E} = \{ 0,1,3q_0,q,2q,3qq_0,q^2 \}$ be the proposed set of orders.
  The subset of $S_{w_8} \smallsetminus I_{\mathcal E}$ of elements minimal with respect to the partial order $\leq_3$ is
  \[
    J = \{ 3q_0+1, q+1, q+3q_0, 3q, 3qq_0+1, 3qq_0+3q_0, 3qq_0+q, 6qq_0 \} .
  \]
  By the $p$-adic criterion, to prove the theorem it will suffice to show that no $j \in J$ is an order of $\mathcal E$.
  In fact, it will be enough to show that no $j \in \{ 3q_0+1, q+1, q+3q_0, 3q \}$ is an order.
  For then we will already know that the six elements of $I_{\mathcal E} \smallsetminus \{3qq_0\}$ are orders.
  Then exactly one of the remaining elements of $J$ is the seventh and final order of $\mathcal E$.
  But each of the remaining elements satisfies $j \geq_3 3qq_0$, so this final order is $3qq_0$.

  That $3q_0+1$ is not an order follows from \eqref{kq0}.
Let $w$ be one of the functions $w_1,w_2,w_3,w_6,w_8$.
Each of these is of the form
\[
w \equiv - h - h^q \mod \overline{\mathbb F}_q(X)^{q^2} ,
\]
where $h = f^{3q_0}(x^q-x)$ and $f \in \{ x,y,z,w_4,w_7 \}$.
Then $D^k w = - D^kh - (D^{k/q}h)^q$ and
\begin{align*}
  D^k h
  &= \sum_{3q_0 i + j = k} (D^i f)^{3q_0} D^j(x^q-x) \\
  &= (x^q-x) (D^{\frac{k}{3q_0}}f)^{3q_0}
  - (D^{\frac{k-1}{3q_0}}f)^{3q_0}
  + (D^{\frac{k-q}{3q_0}}f)^{3q_0} .
\end{align*}
We calculate
\begin{align*}
 D^{3q_0+1} w &= (D^1f)^{3q_0} 				
 &D^{q+3q_0} w &= - (D^1f)^{3q_0} - \ell(D^{q_0+1}f)^{3q_0} \\
 D^{q} w &= (f^q-f)^{3q_0} -\ell(D^{q_0}f)^{3q_0}
 &D^{2q} w &= - (D^{q_0}f)^{3q_0} - \ell(D^{2q_0}f)^{3q_0}
 \addtocounter{equation}{1}\tag{\theequation} \label{Dwi} \\
 D^{q+1} w &= (D^{q_0}f)^{3q_0} 		
 &D^{3q} w &= - (D^{2q_0}f)^{3q_0} .
\end{align*}
Then by using these along with \eqref{nu1} and \eqref{kq0}, one immediately verifies that
\begin{align*}
  \ell D^{q+1}w + D^qw &= \ell^{3q_0} D^{3q_0+1}w  \\
  \ell^{3q_0} D^{q+3q_0}w + D^qw &= 0 \\
  \ell D^{3q}w &= D^{2q}w + D^{q+1}w ,
\end{align*}
and so $q+1$, $q+3q_0$, and $3q$ are not orders of $\mathcal D$.
This completes the proof.
\qedhere

\end{proof}

Up to this point we have shown that the nine numbers
\[
  0,\; 1,\; q_0,\; 2q_0,\; 3q_0,\; q,\; 2q,\; 3qq_0,\; q^2
\]
are orders of $\mathcal D$, and it remains to show that $q+q_0$, $qq_0$, $qq_0+q_0$, $qq_0+q$, and $2qq_0$ are orders.

\begin{proof}[Proof of Theorem \ref{Dthm}]

  Let $I_{\mathcal D}$ be the list of orders of $\mathcal D$ proposed in the statement of the theorem.
  Let $M$ be the 12 by 12 matrix of derivatives $[D^i f_j]$ with $i$ in
  \begin{multline*}
    \{0,1,q_0+1, 2q_0+1, 3q_0+1, q+3q_0+1, 2q+3q_0+1, qq_0+2q+q_0, \\ qq_0+q+2q_0+1, qq_0+2q+3q_0, qq_0+3q+3q_0, 2qq_0+3q_0+1 \}
  \end{multline*}
  and $f_j$ in $\{ 1,x,y,z,w_1,w_2,w_3,w_4,w_7,w_5,w_9,w_{10} \}$.
  The appendix assures that $M$ is upper triangular.
  Moreover, one may check by hand that each diagonal entry is equal to 1, except the last, which is $x^{2q}$.
  Thus there are at least 11 orders which are less than $2qq_0$, and 12 orders which are at most $2qq_0 + 3q_0+1$.
  Since there are 14 orders in total, and we already know that $3qq_0$ and $q^2$ are among them, to prove the theorem it will be enough to show that no element of $(S \smallsetminus I_{\mathcal D}) \cap [0,2qq_0+3q_0+1]$ is an order.

By the $p$-adic criterion, it suffices to check elements of this set which are minimal with respect to $\leq_3$.
These elements comprise the set
\begin{align*}
  J = \{
  &q_0+1,
  3q_0+1,
  q+1,\;
  q+2q_0,\;
  q+3q_0,\;
  2q+q_0,\;
  3q, \;qq_0+1,  \\
  &qq_0+2q_0,\;
  qq_0+3q_0,\;
  qq_0+q+q_0,\;
  qq_0+2q, \;
  2qq_0+q_0 \} .
  \addtocounter{equation}{1}\tag{\theequation} \label{nonorders}
\end{align*}
In fact, it will be enough to demonstrate that each element of $J \smallsetminus \{ 2qq_0+q_0 \}$ is not an order, since $2qq_0 \leq_3 2qq_0 + q_0$.

That $q_0+1$ and $3q_0+1$ are not orders follows from Lemma \ref{smalleps}.
To deal with each remaining ten elements $j \in J$, we give a differential equation
\[
c_jD^jf + \sum_{i < j} c_i D^if = 0, \qquad c_i \in \mathbb F_q(X)
\]
which is satisfied by all $f \in \mathcal B$.
These are listed in the following lemma, and proven in the next section.
This will complete the proof of the theorem.
\qedhere

\end{proof}

\begin{lemma}\label{diffeqns}
The following differential equations are satisfied by each $f \in \mathcal B$:
\begin{enumerate}[label={\normalfont (A\arabic*)}]
\item $\ell^{q_0}D^{q_0+1}f + \ell^{2q_0}D^{2q_0+1}f + \ell^{3q_0}D^{3q_0+1}f = D^qf + \ell D^{q+1}f$
\label{de q+1}
\item $\ell^{q_0}(D^{q+2q_0}f + D^{2q_0+1})f = D^{q+q_0}f + D^{q_0+1}f$
\label{de q+2q0}
\item $D^qf + \ell^{q_0} D^{q+q_0}f + \ell^{2q_0} D^{q+2q_0}f + \ell^{3q_0} D^{q+3q_0}f = 0$
\label{de q+3q0}
\item $\ell D^{2q+q_0}f = D^{q_0+1}f + D^{q+q_0}f$
\label{de 2q+q0}
\item $\ell D^{3q}f = D^{2q}f + D^{q+1}f$
\label{de 3q}
\item $\ell D^{qq_0+1}f + D^{qq_0}f = \ell^q ( \ell^{q_0}D^{2q_0+1}f - D^{q_0+1}f)$
\label{de qq0+1}
\item $\ell^{2q_0}D^{qq_0+2q_0}f - \ell^{q_0}D^{qq_0+q_0}f = \ell^q( D^{q_0+1}f + D^{q+q_0}f ) $
\label{de qq0+2q0}
\item $\ell^{q_0} D^{qq_0+q_0}f + \ell^{2q_0} D^{qq_0+2q_0}f + \ell^{3q_0} D^{qq_0+3q_0}f = \ell D^{qq_0+1}f$
\label{de qq0+3q0}
\item $\ell^{q+q_0+1}D^{qq_0+q+q_0}f = (\ell^q-\ell)\ell^{q_0}D^{qq_0+q_0}f$
\label{de qq0+q+q0}
\item $\ell^{2q}( \ell D^{qq_0+2q}f - D^{qq_0+1}f ) = (\ell^q-\ell)( \ell^qD^{qq_0+q}f + D^{qq_0}f )$ .
\label{de qq0+2q}
\end{enumerate}
\end{lemma}
\begin{proof}
  The fact that $x$, $y$, and $z$ satisfy these equations may be verified without difficulty by hand.
  If $f$ is of type 1, then by consulting the appendix we see that the only equations among \ref{de q+1}--\ref{de qq0+2q} in which nonzero derivatives of $f$ appear are \ref{de q+1}, \ref{de q+3q0}, and \ref{de 3q}, and, keeping in mind that some of the terms are zero for $f$ of type 1, these are the equations which were verified in the proof of Theorem \ref{Ethm}.
  The proof for functions of type 2 is contained in the next section.
  \qedhere
\end{proof}

\section{Verification of Some Differential Equations}

Before proceeding to verify equations \ref{de q+1}--\ref{de qq0+2q} for functions of type 2, we first list a few identities for $w$ of type 1 which will be useful for this task.
For $w$ of type 1, write $w^q-w = f^{3q_0}(b^q-b)$ as in the proof of Theorem \ref{Ethm}.
Then $D^{2q+1}w = (D^{2q_0}f)^{3q_0}$, and so by \eqref{Dwi}, we have
\begin{equation}\label{2q+1}
   \ell D^{2q+1}w + D^{2q}w + D^{q+1}w = 0 .
\end{equation}
Also, from the proof of Theorem \ref{Ethm} we recall that
\begin{align}
  D^qw + \ell D^{q+1}w &= \ell^{3q_0}D^{3q_0+1}w \label{q+1}, \\
  \ell^{3q_0}D^{q+3q_0}w + D^qw &= 0 \label{q+3q0} .
\end{align}

Now let $w$ be of type 2.
From \eqref{ped}, $w$ may be written in the form $w = t_1 - t_2$, where $t_i$ satisfies
\[
  t_i^q - t_i = h_i =  f_i^{q_0}(b_i^q-b_i)  ,
\]
for some $f_i \in \{ w_1,w_2,w_3,w_6\}$ and $b_i \in \{ x,y,z,w_4\}$.
Thus, if one of the desired equations holds for all $t_i$ of this form, then it also holds for $w$.
This will be the case for some but not all of the equations we wish to verify.

Let $t,h,f,b$ be as above.
Then
\[
  D^k h
  = \sum_{q_0 i + j = k} (D^i f)^{q_0} D^j(b^q-b)
\]
and
\[
  D^it = -D^ih - (D^{i/q}h)^q .
\]
Taking into account the supports of $f$ and $b$ and using \eqref{nu1} and \eqref{kq0} to make simplifications, we compute the derivatives of $t$ which appear in equations \ref{de q+1}--\ref{de qq0+2q}:
\begin{align*}
  D^{kq_0+1}t
&= f^{q_0} D^{kq_0+1}b
+ (D^1f)^{q_0} D^{(k-1)q_0 + 1}b , \qquad k=1,2,3 \\
  D^{q} t
&= f^{q_0}D^qb
+ (D^1f)^{q_0} \ell^{q_0} D^q(b^q)
+ (D^{3q_0+1}f)^{q_0} \ell^{q_0+1} D^1b  \\
 D^{q+1}t
&=  f^{q_0}D^{q+1}b
-(D^{3q_0+1}f)^{q_0} \ell^{q_0} D^1b \\
 D^{q+q_0}t
&= f^{q_0}D^{q+q_0}b
- (D^1f)^{q_0}D^q(b^q-b) \\*
&\qquad + (D^{3q_0+1}f)^{q_0} ( \ell^{q_0+1}D^{q_0+1}b - \ell D^1b ) \\
 D^{q+2q_0}t
&= f^{q_0}D^{q+2q_0} b
+ (D^1f)^{q_0}D^{q+q_0}b \\*
&\qquad + (D^{3q_0+1}f)^{q_0} ( \ell^{q_0+1}D^{2q_0+1}b - \ell D^{q_0+1}b ) \\
 D^{q+3q_0}t
&= f^{q_0}D^{q+3q_0} b
+ (D^1f)^{q_0}D^{q+2q_0}b
- (D^{3q_0+1}f)^{q_0} \ell D^{2q_0+1}b  \\
 D^{2q}t
&= f^{q_0}D^{2q}b
+ (D^{3q_0+1}f)^{q_0}\ell^{q_0}D^q(b^q-b) \\
 D^{2q+q_0}t
&= f^{q_0}D^{2q+q_0}b + (D^1f)^{q_0}D^{2q}b \\*
&\qquad - (D^{3q_0+1}f)^{q_0} (\ell^{q_0}D^{q+q_0}b + D^q(b^q-b) ) \\
 D^{3q}t
&= f^{q_0}D^{3q}b
- (D^{3q_0+1}f)^{q_0}\ell^{q_0}D^{2q}b \\
D^{qq_0} t &=
f^{q_0}D^{qq_0}b
+ (D^1f)^{q_0}\ell^{q_0} D^{qq_0}(b^q) \\*
&\qquad - (D^qf)^{q_0} \ell D^1b
- (D^qf^q)^{q_0}\ell^qD^q(b^q) \\
 D^{qq_0+1}t
&= f^{q_0}D^{qq_0+1}b
+ (D^qf)^{q_0} D^1b \\
 D^{qq_0+q_0} t
&= f^{q_0}D^{qq_0+q_0}b
 - (D^1f)^{q_0}D^{qq_0}(b^q-b) \\*
&\qquad  - (D^qf)^{q_0}\ell D^{q_0+1}b
 - (D^{q+1}f)^{q_0}\ell D^1b \\
 D^{qq_0+2q_0} t
&= f^{q_0}D^{qq_0+2q_0}b
+ (D^1f)^{q_0} D^{qq_0 + q_0}b \\*
&\qquad - (D^qf)^{q_0}\ell D^{2q_0+1}b
- (D^{q+1}f)^{q_0}\ell D^{q_0+1}b . \\
 D^{qq_0+3q_0} t
&= f^{q_0}D^{qq_0+3q_0}b - (D^{q+1}f)^{q_0} \ell D^{2q_0+1} b \\
 D^{qq_0+q} t
&= f^{q_0}D^{qq_0+q}b
+ (D^1f)^{q_0}\ell^{q_0}D^{qq_0+q}(b^q) \\*
&\qquad - (D^{3q_0+1}f)^{q_0}\ell^{q_0}( D^{qq_0}b - D^{qq_0}(b^q) )
- (D^qf)^{q_0}D^q(b^q-b) \\*
&\qquad - (D^{q+3q_0}f)^{q_0} \ell D^1b
+ (D^qf^q)^{q_0} D^q(b^q) \\
 D^{qq_0+q+q_0}t
&= f^{q_0}D^{qq_0+q+q_0}b - (D^1f)^{q_0}D^{qq_0+q}(b^q-b) \\*
& - (D^{3q_0+1}f)^{q_0}( \ell^{q_0}D^{qq_0+q_0}b + D^{qq_0}(b^q-b) ) + (D^qf)^{q_0}D^{q+q_0}b \\*
& - (D^{q+1}f)^{q_0}D^q(b^q-b) + (D^{q+3q_0}f)^{q_0}D^{q_0}b - (D^{q+3q_0+1}f)^{q_0}\ell D^1b \\
 D^{qq_0+2q} t
&= f^{q_0} D^{qq_0+2q}b
+ (D^{3q_0+1}f)^{q_0} \ell^{q_0} D^{qq_0+q}(b^q-b)  \\*
&\qquad
- (D^qf)^{q_0}D^{2q}(b^q-b)  - (D^{q+3q_0}f)^{q_0} D^q(b^q-b) .
\end{align*}
We will also use a few of the values $D^ib$, which we collect in the following table:
\begin{equation}\label{btable}
\begin{tabular}{l|lllll}
$i$     &$D^iy$    &$D^iz$       &$D^iw_4$  \\ \hline
$1$     &$x^{q_0}$ &$x^{2q_0}$   &$-x^{2q_0+1}-x^{q_0}y - z$ \\
$q_0+1$ &1         &$-x^{q_0}$   & $x^{q_0+1}-y$ \\
$2q_0+1$&0         &1            & $-x^q$ \\
$q+1$   &0      &0            & $-\ell^{2q_0}$ \\
$q+q_0$ &$-1$   &$x^{q_0}$    & $\ell^{q_0+1} - x^{q_0+1} + y$ \\
$2q$    &0      &0            & $\ell^{2q_0}$ \\
$qq_0+1$ &0      &0            &$-\ell^{q+q_0}$ \\
$qq_0+q_0$ &0      &0            &$-\ell^{q+1}$
\end{tabular}
\end{equation}

Now we verify each of the equations \ref{de q+1}--\ref{de qq0+2q} in turn.
Since each $b_i$ associated with $w$ appears before $w$ in the list
\[
  x,\; y,\; z,\; w_4,\; w_5,\; w_7,\; w_9,\; w_{10},
\]
and since each of \ref{de q+1}--\ref{de qq0+2q} has already been verified for $x$, $y$, and $z$, we may assume by induction that these equations are satisfied by each $b_i$ which appears in the proof.

\begin{proof}[Proof of \ref{de q+1}]
We show that the desired equation holds for all $t$.
Since $D^{3q_0+1}b = 0$, we have
\begin{align*}
 \sum_{k=1}^3 \ell^{kq_0} D^{kq_0+1}t
 &= f^{q_0} \sum_{k=1}^3 \ell^{kq_0} D^{kq_0+1} b + (D^1f)^{q_0} \ell^{q_0} ( D^1b +  \sum_{k=1}^3 \ell^{kq_0} D^{kq_0+1} b ) \\
 &= f^{q_0}( \ell D^{q+1}b + D^qb) + (D^1f)^{q_0} \ell^{q_0} ( D^1b + D^qb + \ell D^{q+1}b ) \\
 &= f^{q_0}( \ell D^{q+1}b + D^qb) + (D^1f)^{q_0} \ell^{q_0} D^q(b^q) \\
 &= D^qt + \ell D^{q+1}t ,
\end{align*}
where we used \eqref{Dqbbq} in the third line.
\end{proof}

\begin{proof}[Proof of \ref{de q+2q0}]
Let
\[
  \Delta_t := \ell^{q_0}(D^{q+2q_0}t + D^{2q_0+1}t) - ( D^{q+q_0}t + D^{q_0+1}t ) .
\]
Here it is not the case that $\Delta_t = 0$ for all $t$, so we need to show that $\Delta_{t_1} = \Delta_{t_2}$ for each $w = t_1-t_2$.
The contribution to $\Delta_t$ of those terms involving $f^{q_0}$ is zero by our assumption on $b$.
Moreover, by \eqref{Dqbbq} the contributed coefficient of $(D^1f)^{q_0}$ is
\begin{align*}
 &\ell^{q_0}D^{q+q_0}b + \ell^{q_0} D^{q_0+1}b + D^q(b^q-b) - D^1b \\
 &\qquad = \ell^{q_0}D^{q+q_0}b + \ell^{q_0} D^{q_0+1}b + \ell D^{q+1}b ,
\end{align*}
which is zero by \eqref{btable}.
Therefore, the only terms giving a contribution to $\Delta_t$ are those involving $(D^{3q_0+1}f)^{q_0}$.
The coefficient of $(D^{3q_0+1}f)^{q_0}$ in $\Delta_t$ is
\[
 \sum_{k=0}^2 \ell^{kq_0+1}D^{kq_0+1}b
 = \ell (D^1b + D^qb + D^{q+1}b)
 = \ell D^q(b^q) ,
\]
where we have used the fact that $D^{3q_0+1}b = 0$, along with \eqref{Dqbbq} and \ref{de q+1}.
With the exception of $w = w_7$ we have $f_i^q - f_i = b_j^{3q_0}(x^q-x)$ for $i\neq j$, so that
\[
(D^{3q_0+1}f_i)^{q_0}
= (D^{3q_0}(b_j^{3q_0}))^{q_0}
= D^q(b_j^q) .
\]
Therefore, for $w \neq w_7$ we have
\[
  \Delta_{t_i} = D^q(b_j^q) \cdot \ell D^q(b_i^q) ,
\]
and so $\Delta_{t_1} = \Delta_{t_2}$ as desired.
Finally, for $w=w_7$ we check that
\[
D^q(y^q) \cdot \ell D^q(y^q)
= x^{2qq_0} \ell
= D^q(z^q) \cdot \ell D^q(x^q) . \qedhere
\]
\end{proof}

\begin{proof}[Proof of \ref{de q+3q0}]
We show that the desired equation holds for all $t$.
Since $D^{q+3q_0}b = 0$, we have
\begin{align*}
  \sum_{k=0}^3 \ell^{kq_0}D^{q+kq_0} t
&= f^{q_0} \sum_{k=0}^3 \ell^{kq_0} D^{q+kq_0}b \\
&\qquad + (D^1f)^{q_0} \ell^{q_0} \sum_{k=0}^3 \ell^{kq_0} D^{q+kq_0}b
+ (D^{3q_0+1}f)^{q_0} \cdot 0
= 0 . \qedhere
\end{align*}
\end{proof}

\begin{proof}[Proof of \ref{de 2q+q0}]
We show that the desired equation holds for all $t$.
The contribution of the terms involving $f^{q_0}$ is zero by our assumption on $b$.
By considering the contribution of terms involving $(D^1f)^{q_0}$ and $(D^{3q_0+1}f)^{q_0}$, it will suffice to show that the equations
\begin{align*}
  D^q(b^q-b) + D^1b &= \ell D^{2q}b  \\
  \ell^{q_0}D^{q_0}b + \ell D^1b &= \ell^{q_0+1}D^{q+q_0}b + \ell D^q(b^q-b)
\end{align*}
hold for $b = x,y,z,w_4$.
These may be rewritten using \eqref{Dqbbq} as
\begin{align*}
  D^{2q}b + D^{q+1}b = 0 \\
  \ell^{q_0}D^{q+q_0}b + \ell D^{q+1}b + \ell^{q_0}D^{q_0+1}b = 0 ,
\end{align*}
and these follow from \eqref{btable}.
\end{proof}

\begin{proof}[Proof of \ref{de 3q}]
We show that the desired equation holds for all $t$.
By our assumption on $b$, the contribution of all terms involving $f^{q_0}$ is zero.
By comparing the coefficients of $(D^{3q_0+1}f)^{q_0}$ and using \eqref{Dqbbq}, it suffices to show that
\[
  \ell D^{2q}b
  = D^1b - D^q(b^q-b)
  = -\ell D^{q+1}b .
\]
This follows from \eqref{btable}.
\end{proof}

\begin{proof}[Proof of \ref{de qq0+1}]
Let
\[
  \Delta_t
  = \ell D^{qq_0+1}t + D^{qq_0}t - \ell^q ( \ell^{q_0}D^{2q_0+1}t - D^{q_0+1}t) .
\]
By our assumption on $b$, the terms in $\Delta_t$ involving $f^{q_0}$ sum to zero.
Moreover, the terms involving $(D^qf)^{q_0}$ also give no contribution to $\Delta_t$.
After some simplification, the sum of the remaining terms is $\ell^q$ times
\[
(D^1f)^{q_0} D^1b
- ((D^1f)^{q_0}D^1b )^q
- (D^1f)^{q_0}
\ell^{q_0} (
D^{q_0 + 1} b
+ (D^{q_0+1} b)^q
) .
\]
Note that $(D^1f)^{q_0} = c^q$, where $f^q-f = c^{3q_0}(x^q-x)$ and $c \in \{x,y,z,w_4\}$.
Thus, $\Delta_t$ is $\ell^q$ times
\[
  c^q D^1b - (c^q D^1b)^q - c^q \ell^{q_0} (
D^{q_0 + 1} b
+ (D^{q_0+1} b)^q) .
\]
That $\Delta_1 = \Delta_2$ may now be checked in each case by using \eqref{btable}.
\end{proof}

\begin{proof}[Proof of \ref{de qq0+2q0}]
Let
\[
  \Delta_t
  = \ell^{2q_0}D^{qq_0+2q_0}t
  - \ell^{q_0}D^{qq_0+q_0}t
  - \ell^q D^{q_0+1}t
  - \ell^qD^{q+q_0}t .
\]
By our assumption on $b$, the contribution to $\Delta_t$ of those terms involving $f^{q_0}$ is zero.
The contribution of the terms involving $(D^1f)^{q_0}$ is
\begin{align*}
&\ell^{2q_0} D^{qq_0+q_0}b
+ \ell^{q_0}D^{qq_0}(b^q-b)
+ \ell^qD^q(b^q-b)
- \ell^qD^1b \\
&\qquad =
  \ell^{2q_0} D^{qq_0+q_0}b
  + \ell^{q_0+1} D^{qq_0+1}b
  + \ell^{q+1} D^{q+1}b .
\end{align*}
This is also zero by \eqref{btable}.
Then using \eqref{q+1}, we may rewrite the sum of the remaining terms as
\begin{multline*}
  \Delta_t
  = (\ell^{3q_0}D^{3q_0+1}f)^{q_0} \ell^{q_0}D^{q_0}b \\
  - (\ell^{3q_0}D^{3q_0+1}f)^{q_0} ( \ell^{2q_0}D^{2q_0}b + \ell^{q_0}D^{q_0}b - \ell D^1b ) \\
  + (\ell D^{q+1}f)^{q_0} ( \ell^{2q_0}D^{2q_0}b + \ell^{q_0}D^{q_0}b - \ell D^1b ).
\end{multline*}
As in the proof of \ref{de q+2q0}, we have
\[
  \ell^{2q_0}D^{2q_0}b + \ell^{q_0}D^{q_0}b - \ell D^1b = \ell D^q(b^q) .
\]
Furthermore, if $c \in \{x,y,z,w_4\}$ with $f^q-f = c^{3q_0}(x^q-x)$, then $(D^{3q_0+1}f)^{q_0} = D^q(c^q)$ and $(D^{q+1}f)^{q_0} = (D^{q_0}c)^q$.
Therefore,
\[
  \Delta_t
  = \ell^{q_0} D^q(c^q)D^{q_0}b
  - \ell D^q(c^q) D^q(b^q)
  + \ell (D^{q_0}c)^q D^q(b^q) .
\]
Now \eqref{btable} may be used to verify in each case that $\Delta_{t_1} = \Delta_{t_2}$.
\end{proof}

\begin{proof}[Proof of \ref{de qq0+3q0}]
By using \eqref{q+1}, the terms on the left hand side of the desired equation may be written as
\begin{align*}
\ell^{q_0}D^{qq_0+q_0} t
&= f^{q_0}\ell^{q_0}D^{qq_0+q_0}b
 - (D^1f)^{q_0} \ell^{q_0}D^{qq_0}(b^q-b) \\
 &\qquad - (\ell^{3q_0}D^{3q_0+1}f)^{q_0} \ell D^1b
 + (D^qf)^{q_0}( \ell^{q_0} D^{q_0}b + \ell D^1b ) \\
\ell^{2q_0}D^{qq_0+2q_0} t
&= f^{q_0}\ell^{2q_0}D^{qq_0+2q_0}b
+ (D^1f)^{q_0} \ell^{2q_0} D^{qq_0 + q_0}b \\
&\qquad + (\ell^{3q_0}D^{3q_0+1}f)^{q_0} \ell^{q_0} D^{q_0}b
+ (D^qf)^{q_0} ( \ell^{2q_0} D^{2q_0}b - \ell^{q_0} D^{q_0}b ) \\
\ell^{3q_0}D^{qq_0+3q_0} t
&= f^{q_0}\ell^{3q_0}D^{qq_0+3q_0}b \\
&\qquad+ (\ell^{3q_0}D^{3q_0+1}f)^{q_0} \ell^{2q_0}D^{2q_0}b
- (D^qf)^{q_0} \ell^{2q_0}D^{2q_0}b .
\end{align*}
By our assumption on $b$, the contribution of the terms involving $f^{q_0}$ is zero.
The terms involving $(D^qf)^{q_0}$ sum to zero.
The sum of the terms involving $(D^{3q_0+1}f)^{q_0}$ is $\ell^q$ times the quantity which was dealt with \ref{de q+2q0}, so by the same argument these give no contribution to $\Delta_{t_1} - \Delta_{t_2}$.
It remains only to show that the terms involving $(D^1f)^{q_0}$ sum to zero, i.e., that
\[
  \ell^{2q_0} D^{qq_0 + q_0}b
  = \ell^{q_0}D^{qq_0}(b^q-b)
  = \ell^{q_0+1} D^{qq_0+1} b .
\]
But this follows from \eqref{btable}.
\end{proof}

\begin{proof}[Proof of \ref{de qq0+q+q0}]
We show that the desired equation holds for all $t$.
Use \eqref{q+1} and \eqref{q+3q0} to write each side of the desired equation as
\begin{multline*}
 \ell^{q_0}D^{qq_0+q_0} t
= f^{q_0}\ell^{q_0}D^{qq_0+q_0}b
- (D^1f)^{q_0} \ell^{q_0}D^{qq_0}(b^q-b) \\
 - \ell^{q} (D^{3q_0+1}f)^{q_0} (b^q-b)
+ (D^qf)^{q_0}( \ell^{q_0} D^{q_0}b + (b^q-b) )
\end{multline*}
and
\begin{align*}
 \ell^{q+q_0} D^{qq_0+q+q_0}t &=
 f^{q_0} \ell^{q+q_0} D^{qq_0+q+q_0}b
- (D^1f)^{q_0} \ell^{q+q_0} D^{qq_0+q}(b^q-b) \\
  &- \ell^q (D^{3q_0+1}f)^{q_0} (
  \ell^{2q_0} D^{qq_0+q_0}b \\
  &\qquad+ \ell^{q_0} D^{qq_0}(b^q-b)
  + \ell^q D^q(b^q-b)
  - \ell D^1b ) \\
 &+ (D^qf)^{q_0} \left(
  \ell^{q+q_0} D^{q+q_0}b
  + \ell^q D^q(b^q-b)
  - \ell^{q_0} D^{q_0}b
  - \ell D^1b
  \right) .
\end{align*}
By our assumption on $b$, the contribution of the terms involving $f^{q_0}$ is zero.
After using \eqref{nu1} to rewrite the coefficients of $(D^1f)^{q_0}$, $\ell^q(D^{3q_0+1}f)^{q_0}$, and $(D^qf)^{q_0}$ completely in terms of derivatives of $b$ and doing some simplification, it will suffice to show that the equations
\begin{align*}
  D^{qq_0+1}b + \ell^q D^{qq_0+q+1}b
  &= 0  \\
  \ell^{q+1} D^{q+1}b
  + \ell^{2q_0} D^{qq_0+q_0}b
  + \ell^{q_0+1} D^{qq_0+1}b
  &= 0  \\
  \ell^{q_0}D^{q_0+1}b
  + \ell D^{q+1} b
  + \ell^{q_0}D^{q+q_0}b
  &= 0
\end{align*}
hold for $b = x,y,z,w_4$.
This is easily verified by consulting \eqref{btable}.
\end{proof}

\begin{proof}[Proof of \ref{de qq0+2q}]
We show that the desired equation holds for all $t$.
By using \eqref{q+3q0} to replace occurrences of $D^{qq_0+3q_0}f$ with $D^{q}f$, we find that
\begin{align*}
 \ell^q D^{qq_0+q} t &=
f^{q_0} \ell^q D^{qq_0+q}b
- (D^1f)^{q_0} \ell^{q_0}D^{qq_0}(b^q) \\
&\qquad
+ (D^{3q_0+1}f)^{q_0} \ell^{q+q_0} D^{qq_0}(b^q-b) \\
&\qquad - (D^qf)^{q_0}( \ell^q D^q(b^q-b) - \ell D^1b ) + (D^qf^q)^{q_0} \ell^q D^q(b^q) \\
 \ell^{2q} D^{qq_0+2q} t &=
f^{q_0} \ell^{2q} D^{qq_0+2q}b
+ (D^{3q_0+1}f)^{q_0} \ell^{2q+q_0} D^{qq_0+q}(b^q-b)  \\
&\qquad
+ (D^qf)^{q_0} ( \ell^q D^q(b^q-b) - \ell^{2q} D^{2q}(b^q-b) ) .
\end{align*}
By our assumption on $b$, the contribution of the terms involving $f^{q_0}$ is zero.
The terms involving $(D^1f)^{q_0}$ and $(D^qf^q)^{q_0}$ also give no contribution.

By comparing the coefficients of $(D^{3q_0+1}f)^{q_0}$ and $(D^qf)^{q_0}$ and doing some minor simplification, it will suffice to show that the equations
\begin{align*}
  \ell^{q+1}D^{qq_0+q}(b^q-b) &= (\ell^q-\ell) D^{qq_0}(b^q-b) \\
  \ell D^{2q}(b^q-b) + D^1b
  &=  D^q(b^q-b)
\end{align*}
hold for $b = x,y,z,w_4$.
Each of these is easily verified using \eqref{nu1} and \eqref{btable}.
\end{proof}

This completes the proof of Lemma \ref{diffeqns}, and hence of Theorem \ref{Dthm}.

\section{Weierstrass points}

As a consequence of Theorem \ref{Dthm}, we determine the Weierstrass points of $\mathcal D$.
Recall that the $\mathcal D$-Weierstrass points are those $P \in X$ satisfying $j_i(P) \neq \epsilon_i(P)$ for some $i$.
These points make up the support of a divisor $R_{\mathcal D}$ with
\[
  \deg R_{\mathcal D} = (2g-2) \sum \epsilon_i + (13+1) m .
\]

Since the sequence $\nu_i(\mathcal D)$ of Frobenius orders differs from $\epsilon_0,\ldots,\epsilon_{13}$, Corollary 2.10 of \cite{SV86} implies that every rational point of $X$ is a $\mathcal D$-Weierstrass point.
We claim that in fact $\operatorname{Supp} R_{\mathcal D} = X(\mathbb F_q)$.

\begin{corollary}
  The set of Weierstrass points of $\mathcal D$ consists of the $\mathbb F_q$-rational points of $X$.
\end{corollary}

\begin{proof}

By Theorem \ref{Dthm}, we have
\[
  \deg R_{\mathcal D}
  =  (2g-2)\sum \epsilon_i + (13+1)m
  = (3qq_0+9q+23q_0+12) N ,
\]
and so it will suffice to show that $v_P(R) = 3qq_0+9q+23q_0+12$ for $P \in X(\mathbb F_q)$.
This will follow from the inequality
\begin{equation} \label{vPR}
  v_P(R_{\mathcal D})
  \geq \sum_{i=0}^r (j_i(P) - \epsilon_i) .
\end{equation}
Since the automorphism group acts doubly transitively on the $\mathbb F_q$-rational points of $X$, it will suffice to show this for the point $P_0$ with $x=y=z=0$.
By expanding out the functions in $\mathcal B$ in terms of $x$, we find that they vanish at $P_0$ to the orders
\begin{align*}
  &j_0 = 0,           &&j_7 = 1+3q_0+2q,        \\
  &j_1 = 1,           &&j_8 = 1+2q_0+q+qq_0,    \\
  &j_2 = 1+q_0,       &&j_9 = 1+3q_0+q+qq_0,    \\
  &j_3 = 1+2q_0,      &&j_{10} = 1+3q_0+2q+qq_0,   \\
  &j_4 = 1+3q_0,      &&j_{11} = 1+3q_0+2q+2qq_0,  \\
  &j_5 = 1+2q_0+q,    &&j_{12} = 1+3q_0+2q+3qq_0,  \\
  &j_6 = 1+3q_0+q,    &&j_{13} = 1+3q_0+2q+3qq_0+q^2.
\end{align*}
Inserting these values into \eqref{vPR} completes the proof.
\qedhere
\end{proof}

The same argument shows that the $\mathcal E$-Weierstrass points are exactly the $\mathbb F_q$-rational points as well.
In this case, $v_P(R_{\mathcal E}) = 3qq_0+4q+12q_0+5$ for $P \in X(\mathbb F_q)$.

\newpage

\section{Appendix}

In the following tables, an asterisk in row $i$ and column $f$ indicates that $i$ is in the set $S_f$ described in section 3.
In particular, $D^i f = 0$ wherever there is a blank entry in the table.
The first table involves functions of type 1 and the second involves functions of type 2.

\setlength\tabcolsep{2pt}

\footnotesize{

\begin{multicols}{2}

\begin{tabular}{|l|cccccc|}
$i$ &$x$ &$w_1$ &$w_2$ &$w_3$ &$w_6$ &$w_8$ \\
\hline
$0$                 &$*$ &$*$ &$*$ &$*$ &$*$ &$*$   \\
$1$                 &$*$ &$*$ &$*$ &$*$ &$*$ &$*$   \\
$3q_0$              &    &$*$ &$*$ &$*$ &$*$ &$*$   \\
$3q_0+1$            &    &$*$ &$*$ &$*$ &$*$ &$*$   \\
$q$                 &    &$*$ &$*$ &$*$ &$*$ &$*$   \\
$q+1$               &    &    &$*$ &$*$ &$*$ &$*$   \\
$q+3q_0$            &    &$*$ &$*$ &$*$ &$*$ &$*$   \\
$q+3q_0+1$          &    &    &$*$ &$*$ &$*$ &$*$   \\
$2q$                &    &    &$*$ &$*$ &$*$ &$*$   \\
$2q+1$              &    &    &    &$*$ &$*$ &$*$   \\
$2q+3q_0$           &    &    &$*$ &$*$ &$*$ &$*$   \\
$2q+3q_0+1$         &    &    &    &$*$ &$*$ &$*$   \\
$3q$                &    &    &    &$*$ &$*$ &$*$   \\
$3q+3q_0$           &    &    &    &$*$ &$*$ &$*$   \\
$3qq_0$             &    &$*$ &$*$ &$*$ &$*$ &$*$   \\
$3qq_0+1$           &    &    &$*$ &$*$ &$*$ &$*$   \\
$3qq_0+3q_0$        &    &    &    &    &$*$ &$*$   \\
$3qq_0+3q_0+1$      &    &    &    &    &$*$ &$*$   \\
\end{tabular}

\begin{tabular}{|l|cccccc|}
$i$ &$x$ &$w_1$ &$w_2$ &$w_3$ &$w_6$ &$w_8$ \\
\hline
$3qq_0+q$           &    &$*$ &$*$ &$*$ &$*$ &$*$   \\
$3qq_0+q+1$         &    &    &$*$ &$*$ &$*$ &$*$   \\
$3qq_0+q+3q_0$      &    &    &    &    &$*$ &$*$   \\
$3qq_0+q+3q_0+1$    &    &    &    &    &$*$ &$*$   \\
$3qq_0+2q$          &    &    &$*$ &$*$ &$*$ &$*$   \\
$3qq_0+2q+1$        &    &    &    &$*$ &$*$ &$*$   \\
$3qq_0+2q+3q_0$     &    &    &    &    &$*$ &$*$   \\
$3qq_0+2q+3q_0+1$   &    &    &    &    &$*$ &$*$   \\
$3qq_0+3q$          &    &    &    &$*$ &$*$ &$*$   \\
$3qq_0+3q+3q_0$     &    &    &    &    &$*$ &$*$   \\
$6qq_0$            &    &    &    &    &$*$ &$*$   \\
$6qq_0+1$          &    &    &    &    &$*$ &$*$   \\
$6qq_0+q$          &    &    &    &    &$*$ &$*$   \\
$6qq_0+q+1$        &    &    &    &    &$*$ &$*$   \\
$6qq_0+2q$         &    &    &    &    &$*$ &$*$   \\
$6qq_0+2q+1$       &    &    &    &    &$*$ &$*$   \\
$6qq_0+3q$         &    &    &    &    &$*$ &$*$   \\
$q^2$               &    &$*$ &$*$ &$*$ &$*$ &$*$   \\
\end{tabular}

\end{multicols}

\begin{multicols}{2}

\begin{tabular}{|l|ccccccc|}
$i$ &$y$ &$z$ &$w_4$ &$w_7$ &$w_5$ &$w_9$ &$w_{10}$ \\
\hline
$0$                 &$*$ &$*$ &$*$ &$*$ &$*$ &$*$ &$*$   \\
$1$                 &$*$ &$*$ &$*$ &$*$ &$*$ &$*$ &$*$   \\
$q_0$               &$*$ &$*$ &$*$ &$*$ &$*$ &$*$ &$*$   \\
$q_0+1$             &$*$ &$*$ &$*$ &$*$ &$*$ &$*$ &$*$   \\
$2q_0$              &    &$*$ &$*$ &$*$ &$*$ &$*$ &$*$   \\
$2q_0+1$            &    &$*$ &$*$ &$*$ &$*$ &$*$ &$*$   \\
$3q_0$              &    &    &    &    &$*$ &$*$ &$*$   \\
$3q_0+1$            &    &    &    &    &$*$ &$*$ &$*$   \\
$q$                 &$*$ &$*$ &$*$ &$*$ &$*$ &$*$ &$*$   \\
$q+1$               &    &    &$*$ &$*$ &$*$ &$*$ &$*$   \\
$q+q_0$             &$*$ &$*$ &$*$ &$*$ &$*$ &$*$ &$*$   \\
$q+q_0+1$           &    &    &$*$ &$*$ &$*$ &$*$ &$*$   \\
$q+2q_0$            &    &$*$ &$*$ &$*$ &$*$ &$*$ &$*$   \\
$q+2q_0+1$          &    &    &$*$ &$*$ &$*$ &$*$ &$*$   \\
$q+3q_0$            &    &    &    &    &$*$ &$*$ &$*$   \\
$q+3q_0+1$          &    &    &    &    &$*$ &$*$ &$*$   \\
$2q$                &    &    &$*$ &$*$ &$*$ &$*$ &$*$   \\
$2q+1$              &    &    &    &    &    &$*$ &$*$   \\
$2q+q_0$            &    &    &$*$ &$*$ &$*$ &$*$ &$*$   \\
$2q+2q_0$           &    &    &$*$ &$*$ &$*$ &$*$ &$*$   \\
$2q+3q_0$           &    &    &    &    &$*$ &$*$ &$*$   \\
$2q+3q_0+1$         &    &    &    &    &    &$*$ &$*$   \\
$3q$                &    &    &    &    &    &$*$ &$*$   \\
$3q+3q_0$           &    &    &    &    &    &$*$ &$*$   \\
$qq_0$              &$*$ &$*$ &$*$ &$*$ &$*$ &$*$ &$*$   \\
$qq_0+1$            &    &    &$*$ &$*$ &$*$ &$*$ &$*$   \\
$qq_0+q_0$          &    &    &$*$ &$*$ &$*$ &$*$ &$*$   \\
$qq_0+q_0+1$        &    &    &$*$ &$*$ &$*$ &$*$ &$*$   \\
\end{tabular}

\begin{tabular}{|l|ccccccc|}
$i$ &$y$ &$z$ &$w_4$ &$w_7$ &$w_5$ &$w_9$ &$w_{10}$ \\
\hline
$qq_0+2q_0$         &    &    &    &$*$ &$*$ &$*$ &$*$   \\
$qq_0+2q_0+1$       &    &    &    &$*$ &$*$ &$*$ &$*$   \\
$qq_0+3q_0$         &    &    &    &    &$*$ &$*$ &$*$   \\
$qq_0+3q_0+1$       &    &    &    &    &$*$ &$*$ &$*$   \\
$qq_0+q$            &$*$ &$*$ &$*$ &$*$ &$*$ &$*$ &$*$   \\
$qq_0+q+1$          &    &    &$*$ &$*$ &$*$ &$*$ &$*$   \\
$qq_0+q+q_0$        &    &    &$*$ &$*$ &$*$ &$*$ &$*$   \\
$qq_0+q+q_0+1$      &    &    &$*$ &$*$ &$*$ &$*$ &$*$   \\
$qq_0+q+2q_0$       &    &    &    &$*$ &$*$ &$*$ &$*$   \\
$qq_0+q+2q_0+1$     &    &    &    &$*$ &$*$ &$*$ &$*$   \\
$qq_0+q+3q_0$       &    &    &    &    &$*$ &$*$ &$*$   \\
$qq_0+q+3q_0+1$     &    &    &    &    &$*$ &$*$ &$*$   \\
$qq_0+2q$           &    &    &$*$ &$*$ &$*$ &$*$ &$*$   \\
$qq_0+2q+1$         &    &    &    &    &    &$*$ &$*$   \\
$qq_0+2q+q_0$       &    &    &$*$ &$*$ &$*$ &$*$ &$*$   \\
$qq_0+2q+2q_0$      &    &    &    &$*$ &$*$ &$*$ &$*$   \\
$qq_0+2q+3q_0$      &    &    &    &    &$*$ &$*$ &$*$   \\
$qq_0+2q+3q_0+1$    &    &    &    &    &    &$*$ &$*$   \\
$qq_0+3q$           &    &    &    &    &    &$*$ &$*$   \\
$qq_0+3q+3q_0$      &    &    &    &    &    &$*$ &$*$   \\
$2qq_0$             &    &$*$ &$*$ &$*$ &$*$ &$*$ &$*$   \\
$2qq_0+1$           &    &    &$*$ &$*$ &$*$ &$*$ &$*$   \\
$2qq_0+q_0$         &    &    &    &$*$ &$*$ &$*$ &$*$   \\
$2qq_0+q_0+1$       &    &    &    &$*$ &$*$ &$*$ &$*$   \\
$2qq_0+2q_0$        &    &    &    &    &$*$ &$*$ &$*$   \\
$2qq_0+2q_0+1$      &    &    &    &    &$*$ &$*$ &$*$   \\
$2qq_0+3q_0$        &    &    &    &    &    &    &$*$   \\
$2qq_0+3q_0+1$      &    &    &    &    &    &    &$*$   \\
\end{tabular}
\end{multicols}

}

\bibliography{Ree_orders}{}
\bibliographystyle{plain}

\end{document}